\documentclass[english,reqno]{amsart}

\def\MR#1{}
\usepackage{amsmath}
\usepackage{amssymb} 
\usepackage{enumitem} 
\usepackage{mathtools} 
\usepackage[table]{xcolor} 
\usepackage[all]{xy} 
\usepackage{tikz} 
\usepackage{tikz-cd}
\usepackage{indentfirst} 
\usepackage{babel} 
\usepackage{setspace} 

\usepackage[colorlinks,linkcolor=red,anchorcolor=green,citecolor=blue]{hyperref} 
\hypersetup{linktocpage = true} 

\usepackage{rotating} 

\usepackage{ytableau} 
\usepackage{longtable} 
\newcolumntype{M}[1]{>{\centering\arraybackslash}m{#1}} 

\usepackage{mathpazo}
\usepackage{mathrsfs}
\DeclareFontFamily{OMS}{rsfs}{\skewchar\font'60}
\DeclareFontShape{OMS}{rsfs}{m}{n}{<-5>rsfs5 <5-7>rsfs7 <7->rsfs10 }{}
\DeclareSymbolFont{rsfs}{OMS}{rsfs}{m}{n}
\DeclareSymbolFontAlphabet{\scr}{rsfs}
\DeclareSymbolFontAlphabet{\scr}{rsfs}

\usepackage[T1]{fontenc}

\newcommand{\p}[0]{{\mathbb P}}

\renewcommand{\phi}{\varphi}

\newcommand{\D}{\Delta}

\newcommand{\cF}{\mathcal{F}}
\newcommand{\cG}{\mathcal{G}}

\newcommand{\sO}{\mathscr{O}}

\newtheorem{thm}{Theorem}[section]

\newtheorem{lemma}[thm]{Lemma}
\newtheorem{cor}[thm]{Corollary}

\newtheorem{prop}[thm]{Proposition}

\newtheorem*{defn*}{Definition}
\newtheorem*{thm*}{Theorem}
\theoremstyle{definition}

\newtheorem{defn-thm}[thm]{Definition-Theorem} 
\newtheorem{defn-lemma}[thm]{Definition-Lemma}

\theoremstyle{remark}
\newtheorem{rem}[thm]{Remark}

\newtheorem*{not-and-def}{Notation and definitions}

\numberwithin{equation}{section}

\begin{document}

\title{On generic  nefness of tangent sheaves}

\author{Wenhao Ou}

\address{Wenhao Ou, Institute of Mathematics, Academy of Mathematics and Systems Science, Chinese Academy of Sciences, Beijing, 100190, China}
\email{wenhaoou@amss.ac.cn}

\subjclass[2010]{14E99}

\begin{abstract}
We show that the tangent bundle of a projective manifold with nef anticanonical class is generically nef. That is, its restriction to a  curve  cut out by general sufficiently ample divisors is a nef vector bundle. This confirms a conjecture of Peternell.  As a consequence, the second Chern class of such a manifold has non-negative intersections with ample divisors. We also investigate under which conditions these positivities are strict, and answer a question of Yau.
\end{abstract}

\maketitle
 \setcounter{tocdepth}{1}
\tableofcontents

\section{Introduction}
From the viewpoint of the minimal model program, complex projective manifolds $X$ could be birationally classified   according to the sign of the canonical class $K_X$. It is natural to ask how far we can lift the positivity (or the negativity) of $K_X$ to the cotangent sheaf $\Omega_X^1$. The following two theorems were due to Miyaoka (See \cite[Corollary 6.4]{Miy87}).

\begin{thm}
\label{thm-Miya}
Let $X$ be a complex projective manifold   such that $K_X$ is pseudoeffective. Then the sheaf $\Omega_X^1$ is generically nef. That is, $\Omega_X^1|_C$  is a nef vector bundle for any Mehta-Ramanathan-general curve  $C$.  
\end{thm}

\begin{thm}
\label{thm-Miya2}
Let $X$ be a complex projective manifold of dimension $n$ such that  $K_X$ is nef. Then for any ample  divisors $H_1,...,H_{n-2}$, we have $$c_2(\Omega_X^1)\cdot H_1\cdot \cdots \cdot H_{n-2}\geqslant 0,$$ where $c_2$  stands for the second Chern class. 
\end{thm}

A Mehta-Ramanathan-general curve $C$ is the complete intersection of $n-1$ sufficiently ample divisors in general positions, where $n$ is the dimension of $X$.  In this  paper, we are interested in complex projective varieties with nef anticanonical class $-K_X$. These varieties have been broadly studied,  \textit{e.g.}  \cite{DPS93}, \cite{DPS94}, \cite{DPS96}, \cite{Zhang96}, \cite{Pau97}, \cite{DPS01}, \cite{Zhang05}, \cite{Pau12}, \cite{Pet12},   \cite{CDP15}, \cite{CH17} and \cite{Cao19}.   Based on their works, we prove the following theorem.

\begin{thm}
\label{thm-main}
Let $X$ be a complex projective variety of dimension at least $2$ with $\mathbb{Q}$-factorial log canonical singularities. Assume that  that $-K_X$ is nef.  Then the reflexive tangent sheaf $T_X$ is generically nef. That is, $T_X|_C$  is a nef vector bundle for a Mehta-Ramanathan-general curve  $C$.  
\end{thm}

Theorem \ref{thm-main} was conjectured  by Peternell in \cite[Conjecture 1.5]{Pet12}. Furthermore, if $C$ is  the intersection of $n-1$ sufficiently ample divisors of the same class, then Cao  (see \cite[Theorem 1.2]{Cao13}) and Guenancia (see \cite[Theorem C]{Guenan16}) proved the  nefness of $T_X|_C$  independently by analytic methods.

We   note that  Theorem \ref{thm-main} does not hold if we only assume that  $-K_X$ is pseudoeffective. The following example could be found in  \cite[Example 4.14]{DPS01}.  Let $Y$ be a curve of genus   $g\geqslant 2$ and let $L$ be a line bundle on $Y$ of degree smaller than $2-3g$. Let $X=\mathbb{P}(\sO_Y\oplus L)$ be a surface and let $f \colon X\to Y$ be the natural projection. Then, on the one hand, $-K_X$ is effective. On the other hand, we have a natural surjective morphism $T_X\to p^*T_Y$. Since  $\Omega^1_Y$ is ample, we deduce that  $T_X|_C$ is not nef if $C\subseteq X$ is a general very ample divisor.

We also remark that  on cannot replace a Mehta-Ramanathan-general curve by a movable curve in Theorem \ref{thm-main}.   
Actually, it was shown in \cite[Theorem 7.7]{BDPP13} that  if $X$ is  a   K3-surface or a   Calabi-Yau threefold, then there is a dominant family of    curves $(C_t)_{t\in T}$ such that $T_X|_{C_t}$ is not nef for general $t\in T$.

For movable curves, we prove the following theorem, which  implies Theorem \ref{thm-main} by using Mehta-Ramanathan theorem (see \cite[Theorem 6.1]{MR82}).

\begin{thm}
\label{thm-main-mov}
Let $X$ be a complex projective variety  with $\mathbb{Q}$-factorial log canonical singularities. Assume  that $-K_X$ is nef. Let $\alpha$ be a movable  class of curves. Then for any non-zero torsion-free quotient sheaf $Q$ of $T_X$, we have $\alpha\cdot c_1(Q)\geqslant 0$, where $c_1$ stands for the first Chern class.
\end{thm}

As a corollary, we can obtain the following result, which extends  \cite[Theorem 1.2]{Xie04} to varieties of any  dimension. 
It can be viewed as a dual version of Theorem \ref{thm-Miya2}. 

\begin{cor}
\label{cor-c2}
Let $X$ be a normal complex projective variety of dimension $n$ with nef anticanonical class $-K_X$. Assume that $X$ has $\mathbb{Q}$-factorial log canonical singularities and is smooth in codimension $2$.  Then for any nef divisors $H_1,...,H_{n-2}$, we have $$c_2(T_X)\cdot H_1\cdot \cdots \cdot H_{n-2}\geqslant 0.$$ 
\end{cor}

An orbifold version of Theorem \ref{thm-Miya} was established in  \cite[Theorem 2.1]{CP15a}: if $(X,\D)$ is a projective $\mathbb{Q}$-factorial log canonical pair with $K_X+D$ pseudoeffective, then the orbifold cotangent sheaf $\Omega^1(X,\D)$ is $\pi$-generically nef for any adapted Kawamata finite cover $\pi \colon Z\to X$. In our situation, we can also prove the following orbifold version of Theorem \ref{thm-main}. 

\begin{thm}
\label{thm-main-orbifold}
Let $(X,\D)$ be a complex projective  $\mathbb{Q}$-factorial log canonical pair such that $-(K_X+\D)$ is nef. Then the orbifold tangent sheaf $T(X,\D)$ is $\pi$-generically nef for any adapted Kawamata finite cover $\pi \colon Z\to X$.
\end{thm}

It is natural to ask under which conditions the positivities in the theorems above are strict. In the second part of the paper, we   prove the the following two theorems.

\begin{thm}
\label{thm-gen-ample}
Let $X$ be a smooth  projective complex manifold of dimension $n\geqslant 2$ with nef anticanonical class $-K_X$. Then the following properties are equivalent:
\begin{enumerate}
\item $X$ is rationally connected;
\item $T_X$ is generically ample, that is, $T_X|_C$  is an ample vector bundle for any Mehta-Ramanathan-general curve  $C$.  
\end{enumerate}
\end{thm}

The next theorem was pointed out to the author by Junyan Cao,   and the idea of the proof goes back to Andreas  H\"oring.

\begin{thm}
\label{thm-classification}
Let $X$ be a smooth complex projective manifold of dimension $n$ with nef anticanonical class $-K_X$. Then the following properties are equivalent:
\begin{enumerate}
\item there are ample  divisors $H_1,...,H_{n-2}$ such that $$c_2(T_X)\cdot H_1\cdot \cdots \cdot H_{n-2} = 0;$$
\item there is a finite \'etale cover $\tilde{X}\to X$ such that  $\tilde{X}$  is either isomorphic  to an abelian variety or isomorphic to a $\p^1$-bundle over an abelian variety.
\end{enumerate}
\end{thm}

This theorem answers a question of Yau in the   case of projective manifolds (see \cite[Problem 66]{Yau93}). We note that one could not expect the $\p^1$-bundle   to be trivial, see for example \cite[Example 3.5]{DPS94}.

There are two main ingredients for the proof of Theorem \ref{thm-main-mov}. The first one is the following theorem on algebraicity of foliations, due to  Campana and P\u aun  (See \cite[Theorem 1.1]{CP19}). 

\begin{thm}
\label{thm-CP-foli}
Let $X$ be a projective manifold. Let $\alpha$ be any movable curve class. Let $\cF\subset T_X$ be a foliation. Assume that the slope, with respect to $\alpha$, of any non-zero torsion-free quotient of $\cF$  is strictly positive. Then $\cF$ is  algebraically integrable. That is, $\cF$ is the foliation induced by some rational dominant map $f \colon X\dashrightarrow Y$.  Moreover, general leaves of $\cF$ are rationally connected.
\end{thm}

Another crucial theorem is the following one, which is a refined version of a theorem of Chen and  Zhang (see \cite[Main Theorem]{CZ13}).

\begin{thm}
\label{thm-gen-positive-base-anticanonical}
Let $(X,D)$ be a projective $\mathbb{Q}$-factorial log canonical pair with $-(K_X+D)$ nef. Let $f \colon X\dashrightarrow Y$ be a rational dominant map with $0<\dim Y< \dim X$. Let $\cF$ be the foliation induced by $f$. Then $ K_{\cF} - K_X-D_{ver} $ is pseudo-effective, where $D_{ver}$ is the vertical part of $D$ over $Y$, and $K_{\cF}$ is the canonical class of $\cF$.
\end{thm} 

We recall that, if $f \colon  X \dashrightarrow Y $ is a rational dominant map between two varieties and if $\D$ is a prime divisor in $X$, then $\D$ is said to be horizontal over $Y$ if its strict transform in the graph of $f$ dominates $Y$. Otherwise, $\D$ is said to be vertical over $Y$.

Let us sketch the proof of the Theorem \ref{thm-main-mov}. Our idea is inspired by Peternell's proof in the case of rational surfaces (see \cite[Theorem 5.9]{Pet12}). We  assume by contradiction that there is some movable class $\alpha$ and some  torsion-free quotient $T_X\to Q$ such that $\alpha\cdot c_1(Q) <0$. Then  we can find  a suitable subsheaf, namely $\cF$, in the Harder-Narasimhan semistable filtration of $T_X$ such that   $\alpha\cdot c_1(T/\cF)$ is negative, and that $\cF$ is a foliation on $X$.  In particular, $$\alpha\cdot K_{\cF}< \alpha\cdot K_X.$$ Moreover, by using Theorem \ref{thm-CP-foli},  we can show that $\cF$ is induced by a rational dominant map $f \colon X\dashrightarrow Y$. Then from Theorem \ref{thm-gen-positive-base-anticanonical}, we obtain that $$C\cdot K_{\cF}\geqslant C\cdot K_X.$$ This is a contradiction.

The present paper is organized as follows. After recalling some basic results in Section \ref{section:pre}, we will prove  Theorem \ref{thm-gen-positive-base-anticanonical},  Theorem \ref{thm-main-mov}, Theorem  \ref{thm-main} and  Corollary  \ref{cor-c2} successively in Section \ref{section:positive-tangnet}. 
The orbifold case will be treated in Section \ref{An orbifold version of generic nefness}. 
Finally, we will prove Theorem \ref{thm-gen-ample}  and Theorem \ref{thm-classification} from  Section \ref{section:rc} to Section \ref{section:Yau}.\\

\noindent \textbf{Acknowledgment.} The author would like to express his gratitude to St\'ephane Druel and Burt Totaro for reading the preliminary version of this paper  and warm encouragement.  He is   grateful to Junyan Cao for pointing out Theorem \ref{thm-classification} to him. He would also like to thank  Jun Li, Chen Jiang, Claire Voisin, Yuan Wang and Jian Xiao for general discussions.



\section{Preliminary}
\label{section:pre}
We first collect some notation and  elementary result in this section. Throughout this paper, we will work over $\mathbb{C}$, the field of complex numbers.

\subsection{Slope semistability and Harder-Narasimhan  filtrations}
\label{Slope semistability and Harder-Narasimhan  filtrations}
Let $X$ be a  normal projective variety and let $\alpha$ be a movable curve class.  Assume that either $X$ is $\mathbb{Q}$-factorial or $\alpha$ is the class of a complete intersection of basepoint-free Cartier divisors.  Then for any torsion-free coherent sheaf $E$ of  positive rank on $X$,  the slope of $E$ with respect to $\alpha$ is the number $$\mu_{\alpha}(E)=\frac{\alpha \cdot c_1(E) }{\mathrm{rank}\, E}.$$  
The maximal slope is defined as follows, $$\mu_{\alpha,max}(E)=\sup \{\mu_{\alpha}(F) \ |\  F\ \mbox{is a  non-zero saturated subsheaf of } E\}.$$ The supremum is in fact a maximum. 
The sheaf $E$ is called $\alpha$-semistable   (or just semistable if there is no ambiguity) if $\mu_{\alpha,max} (E)= \mu_{\alpha}(E)$. If $E$ is not semistable, then there is a unique maximal subsheaf $F$ of $E$ such that $\mu_{\alpha,max} (E)= \mu_{\alpha}(F)$. This $F$ is called the maximal destabilizing subsheaf and is automatically semistable.

There is a unique filtration, called the Harder-Narasimhan semistable filtration, of saturated subsheaves, $$0=E_0\subsetneq E_1 \subsetneq \cdots \subsetneq E_r=E$$ such that $E_i/E_{i-1}$ is the maximal destabilizing subsheaf of $E/E_{i-1}$  for all $i\in \{1,..., r\}$ and that the sequence $(\mu_{\alpha}(E_i/E_{i-1}))_{i\in \{1,..., r\}}$ is strictly decreasing.

The minimal slope 
is defined as follows,  $$\mu_{\alpha,min}(E)=\inf \{\mu_{\alpha}(Q) \ |\  Q\ \mbox{is  a  non-zero  torsion-free quotient sheaf  of  } E\}.$$  This infimum is also a minimum. Indeed, we have $\mu_{\alpha,min}(E)= \mu_{\alpha}(E/E_{r-1})$, where $E_{r-1}$ is the saturated subsheaf defined in the Harder-Narasimhan  filtration above.

\begin{lemma}
\label{lem-invol-subsheaf}
Let $X$ be a normal projective  variety and let $\alpha$ be a movable curve class in $X$. 
Assume that either $X$ is $\mathbb{Q}$-factorial or $\alpha$ is the  complete intersection class of basepoint-free Cartier divisors.
Let $E$ be a non-zero torsion-free  sheaf on $X$ such that  $\mu_{\alpha}(E)\geqslant 0$ and $\mu_{\alpha,min}(E) <0$.  Let $$0=E_0\subsetneq E_1 \subsetneq \cdots \subsetneq E_r=E$$ 
be the Harder-Narasimhan semistable filtration  with respect to $\alpha$.   Then there is some  $k\in \{1,...,r-1\}$ such that  $\mu_{\alpha}(E/E_k) <0$ and  $\mu_{\alpha,min}(E_k) = \mu_{\alpha}(E_k/E_{k-1})>0$.
\end{lemma}

\begin{proof}
We have $\mu_{\alpha}(E/E_{r-1})=\mu_{\alpha,min}(E)<0$. Let $k$ be the smallest integer in $\{0,...,r-1\}$ such that $\mu_{\alpha}(E/E_{k})<0$. Since $\mu_{\alpha}(E)\geqslant 0$, we know that $k\geqslant 1$. We consider the following exact sequence $$0\to E_{k}/E_{k-1}\to E/E_{k-1} \to E/E_{k} \to 0.$$ By the definition  of $k$, we have $\mu_{\alpha}(E/E_{k-1})\geqslant 0$. Thus $\mu_{\alpha}(E_k/E_{k-1})>0$.  
\end{proof}

\subsection{Foliations and relative tangent sheaves}

Let $X$ be a normal variety of dimension at least $2$ and let $T_X=(\Omega_X^1)^*$ be the reflexive tangent sheaf. A foliation $\cF$ on $X$ is a saturated subsheaf of $T_X$ which is closed under Lie brackets.  The canonical class $K_{\cF}$ of $\cF$ is a Weil divisor such that $$\sO_{X}(- K_{\cF}) \cong  \mathrm{det}\, \cF,$$  where $\mathrm{det}\, \cF$ is the reflexive hull of the top wedge product of  $\cF$. We say that $\cF$ is  algebraically integrable  if the dimension of the Zariski closure of a general leaf of $\cF$ is equal to the rank of $\cF$.

Typical examples of foliations are relative tangent sheaves. We consider a rational dominant map $f \colon X\dashrightarrow Y$ between normal varieties.  
Let $V$ be the smooth locus of $Y$. Let $U$ be a non-empty smooth open subset of $X$ such that    $f|_U$ is regular and $f(U)\subset V$.  
The relative tangent sheaf $T_{U/V}$ of $f|_U \colon U\to V$ is defined as the kernel of the natural differential map $$\mathrm{d}f|_U \colon T_U\to f^*T_V.$$ 
There is a unique saturated subsheaf $T_{X/Y}$ of the reflexive tangent sheaf $T_X$ such that $T_{X/Y}|_U=T_{U/V}$. We call $T_{X/Y}$ the relative tangent sheaf of $f \colon X\dashrightarrow Y$. It is a foliation on $X$.  We note that a foliation  is  algebraically integrable  if and only if it is induced by some rational dominant map   (see \textit{e.g.} \cite[Lemma 3.2]{AD13}).

The following proposition extends \cite[Theorem 1.1]{CP19} to singular varieties.

\begin{prop}  
\label{prop-CP-alg-foliation-sing}
Let $X$ be a projective normal $\mathbb{Q}$-factorial variety.  Let  $\alpha$ be a movable curve class. Assume that  $\cF$ is a saturated subsheaf of $T_X$ such that 
\begin{enumerate}
\item $\mu_{\alpha,min}(\cF) > 0$,
\item $2\mu_{\alpha,min}(\cF) >  \mu_{\alpha,max} (T_X/\cF).$
\end{enumerate}
Then $\cF$ is an algebraically integrable  foliation.  Moreover, general leaves of $\cF$ are rationally connected.
\end{prop}

\begin{proof}
Let $r \colon X'\to X$ be a resolution of singularities, and let $\alpha'=r^*\alpha$ be the numerical pull-back  such that $$\alpha' \cdot \beta' =\alpha \cdot r_*\beta'$$ for any divisor class $\beta'$ on $X'$ (see \cite[Construction A.15]{GKP14}).  Since $\alpha$ is movable, so is   $\alpha'$  by \cite[Lemma A.17]{GKP14}.  

If $\cF'$ is the saturated subsheaf of $T_{X'}$ induced by $\cF$, then   $$\mu_{\alpha',min}(\cF')  =\mu_{\alpha,min}(\cF) > 0,$$ and $$2\mu_{\alpha',min}(\cF')  = 2 \mu_{\alpha,min}(\cF) >  \mu_{\alpha,max} (T_{X}/\cF)=\mu_{\alpha',max} (T_{X'}/\cF').$$ Hence $\cF'$ is an  algebraically integrable  foliation  by \cite[Theorem 1.4]{CP19}. Moreover, general leaves of $\cF'$ are rationally connected.  The proposition then follows from the property that $\cF\cong (r_*\cF')^{**}$.
\end{proof}

\section{Positivity of tangent sheaves}
\label{section:positive-tangnet}

In this section, we will successively  prove Theorem \ref{thm-gen-positive-base-anticanonical},  Theorem \ref{thm-main-mov}, Theorem  \ref{thm-main} and  Corollary  \ref{cor-c2}.  
We will need the following statement, which is a special case of a result of Druel (see \cite[Proposition 4.1]{Druel2017}).

\begin{prop} 
\label{prop-pseuf}
Let $f \colon X\to Y$ be a surjective morphism between normal projective varieties. 
Assume that $X$ is $\mathbb{Q}$-factorial. 
Let $\D$ be an effective $\mathbb{Q}$-divisor in $X$. 
Assume that the pair $(F,\D|_F)$ is log canonical, where $F$ is a general fiber of $f$. 
If there is some positive integer $m$ such that  $m(K_X+\D)$ is Cartier and that $h^0(F, \sO_F(m(K_X+\D)|_F))>0$, then $K_{\cF}+\D$ is pseudo-effective, where $\cF$ is the foliation induced  by $f$.  
\end{prop}


\begin{proof}[{Proof of Theorem \ref{thm-gen-positive-base-anticanonical}}]

There is a log resolution $\pi \colon Z\to X$ of $(X,D)$ such that the induced map $g \colon Z\to Y$ is a morphism. By blowing up $Y$ and $Z$ if necessary, we may assume that $Y$ is smooth.
We write $$K_Z+D_Z\sim_{\mathbb{Q}} \pi^*(K_X+D)+E,$$ where $D_Z$ and $E$ are effective $\mathbb{Q}$-divisors without common components. Moreover $E$ has $\pi$-exceptional support. Since $(X,D)$ is log canonical, so is the pair $(Z,D_Z)$.

\centerline{
\xymatrix{
  Z \ar[d]^{g}   \ar[r]^{\pi} & X\\
 Y  &
}
}

Let $L$ be an ample divisor in $Z$ and let $\delta >0$ be a rational number. Then $-\pi^*(K_X+D)+\delta L$ is ample. 
We can then choose an effective $\mathbb{Q}$-divisor $$A\sim_{\mathbb{Q}}-\pi^*(K_X+D)+\delta L$$ such that $(Z,D_Z+A)$ is a log canonical. We have $$K_{Z}+ D_Z+A \sim_{\mathbb{Q}} E+\delta L.$$ Let $\D=D_Z+A-\pi_*^{-1}D_{ver}$. Then $\D$ is effective, and  we have \begin{equation}
K_{Z}+ \D \sim_{\mathbb{Q}} E+\delta L-\pi_*^{-1}D_{ver}. 
\end{equation}

Let $G$ be a general fiber of $g$. We claim that $m(K_{Z}+\D)|_G$ has non-zero global sections for large enough and sufficiently divisible integer $m$. Indeed, by the previous equation, we have  $$(K_{Z}+\D)|_G \sim_{\mathbb{Q}} (E+\delta L-\pi_*^{-1}D_{ver})|_G \sim_{\mathbb{Q}} (E+\delta L)|_G.$$ The right-hand-side above is a big divisor. 
Hence $m(K_{Z}+\D)|_G$ has non-zero global sections for large enough and sufficiently divisible integer $m$.

By Proposition \ref{prop-pseuf}, we obtain that $K_{\cG}+\D$ is pseudoeffective, where $\cG$ is the foliation induced by $g$.  Hence 
\begin{eqnarray*}
E+\delta L+(K_{\cG}-K_Z-\pi_*^{-1}D_{ver})  &\sim_{\mathbb{Q}}& E +\delta L - \pi_*^{-1}D_{ver} +(K_{\cG}-K_Z)\\
&\sim_{\mathbb{Q}}& K_{Z}+\D  + (K_{\cG}-K_Z)\\
& \sim_{\mathbb{Q}}& K_{\cG}+\D 
\end{eqnarray*}
is pseudoeffective.

Since this is true for arbitrary $\delta>0$, we obtain that $$E+(K_{\cG}-K_Z-\pi_*^{-1}D_{ver})$$ is pseudoeffective. 
Thus $$K_{\cF}-K_X- D_{ver}=\pi_*(E+(K_{\cG}-K_Z-\pi_*^{-1}D_{ver}))$$ is pseudoeffective.
\end{proof}


\begin{proof}[{Proof of Theorem \ref{thm-main-mov}}]
Assume by contradiction that  $\mu_{\alpha,min}(T_X) < 0$. In particular, $T_X$ is not $\alpha$-semistable. Let $$0=E_0\subsetneq E_1 \subsetneq \cdots \subsetneq E_r=T_X$$ with $r\geqslant 2$ be the Harder-Narasimhan semistable filtration.  
Then by Lemma \ref{lem-invol-subsheaf}, there is some $k\in \{1,...,r-1\}$ such that  $\mu_{\alpha}(T_X/E_k) <0$ and $$\mu_{\alpha,min}(E_k) = \mu_{\alpha}(E_k/E_{k-1})>0.$$  We  have   the following inequality, 
\begin{eqnarray*}
2\mu_{\alpha,min}(E_k) &=& 2\mu_{\alpha}(E_k/E_{k-1}) > \mu_{\alpha}(E_k/E_{k-1}) \\
  &>& \mu_{\alpha}(E_{k+1}/E_{k}) = \mu_{\alpha,max}(T_X /E_{k}).
\end{eqnarray*}  Hence $E_k$ is an algebraically integrable  foliation  by Proposition \ref{prop-CP-alg-foliation-sing}.  We denote $E_k$ by $\cF$. Then $$\alpha\cdot  K_{\cF} < \alpha\cdot K_X $$ for $\mu_{\alpha}(T_X/E_k) <0$.

Since $\cF$ is  algebraically integrable, there is a rational dominant map $f \colon X\dashrightarrow Y$   such that $\cF$ is induced by $f$. Moreover, since $\cF$ is a non-zero proper subsheaf, we have $0 < \dim Y<\dim X$. Hence Theorem \ref{thm-gen-positive-base-anticanonical} shows that $$\alpha\cdot  K_{\cF} \geqslant \alpha\cdot K_X.$$ This is a contradiction. 
\end{proof}


\begin{proof}[{Proof of Theorem \ref{thm-main}}]
Let $C$ be a  Mehta-Ramanathan-general curve.  Let  $$0=E_0\subsetneq E_1 \subsetneq \cdots \subsetneq E_r=T_X$$ be  the Harder-Narasimhan filtration with respect to the class  $\alpha$ of $C$. Then, by  Mehta-Ramanathan Theorem (see \cite[Theorem 6.1]{MR82}), the restriction $$0=E_0|_C\subsetneq E_1|_C \subsetneq \cdots \subsetneq E_r|_C=T_X|_C$$ is the Harder-Narasimhan filtration for $T_X|_C$. In particular, we have $$\mu_{min}(T_X|_C)=\mu_{\alpha, min}(T_X)\geqslant 0.$$  Thus $T_X|_C$ is nef.
\end{proof}


To prove Corollary \ref{cor-c2}, we first recall that a non-zero torsion-free sheaf $E$ on a projective manifold $X$ of dimension $n$ is said to be generically $(H_1,...,H_{n-2})$-semipositive for some ample divisors $H_1,...,H_{n-2}$ if for each nef divisor $D$, we have $\mu_{\alpha,min}(E)\geqslant 0$, where $\alpha$ is the class of $D\cdot H_1\cdot \cdots \cdot H_{n-2}$. 
In particular,  if $E$ is generically nef, then it is generically $(H_1,...,H_{n-2})$-semipositive for any ample divisors $H_1,...,H_{n-2}$.

\begin{proof}[{Proof of Corollary  \ref{cor-c2}}]  By Theorem \ref{thm-main-mov}, we see that  the tangent sheaf $T_X$ is generically $(H_1,...,H_{n-2})$-semipositive.  Since $-K_X$ is nef,  Miyaoka inequality  (see \cite[Theorem 6.1]{Miy87}) shows that  $$c_2(T_X)\cdot H_1\cdot \cdots \cdot H_{n-2}\geqslant 0.$$
\end{proof}

\section{An orbifold version of generic nefness}
\label{An orbifold version of generic nefness}

In this section, we will prove Theorem \ref{thm-main-orbifold}. We  refer to \cite[Section 1]{CP15a} for detailed notion  of orbifolds. Let  $(X,D)$ be a projective $\mathbb{Q}$-factorial log canonical  pair of dimension $n$, where $D$ is an effective $\mathbb{Q}$-divisor.  
Let $\pi \colon Z\to X$ be a  Kawamata finite cover adapted to $(X,D)$.  Let $H_1,...,H_{n-1}$ be ample divisors in $X$. Then the orbifold cotangent sheaf $\Omega^1(X,D)$ (respectively the orbifold tangent sheaf $T(X,D)$) is said to be $\pi$-generically semipositive with respect to $H_1,...,H_{n-1}$ if for any non-zero torsion-free quotient $Q$ of $\pi^*\Omega^1(X,D)$ (respectively of $\pi^*T(X,D)$), we have 
$$c_1(Q) \cdot \pi^* H_1\cdot \cdots \cdot \pi^* H_{n-1}\geqslant 0.$$  
We say  that $\Omega^1(X,D)$ (respectively $T(X,D)$) is $\pi$-generically nef if it is $\pi$-generically semipositive with respect to any ample divisors $H_1,...,H_{n-1}$ in $X$. In particular,   if $D$ is an integral divisor and if $\pi$ is the identity map, then  $\pi$-generic nefness is the same as generic nefness by Mehta-Ramanathan theorem.

In order to prove Theorem \ref{thm-main-orbifold}, we will need the following version of  \cite[Theorem 1.4]{CP19} for singular spaces.   

\begin{thm}  
\label{thm-CP-alg-foliation-sing-orbifold}
Let $(X,D)$ be a projective  $\mathbb{Q}$-factorial log canonical pair. Let $\pi \colon Z\to X$ be a finite cover adapted to $(X,D)$.  Let $H_1, ... , H_{n-1}$ be  very ample divisors  in $X$ and let $\alpha$ be the class of $\pi^*H_1\cdot \cdots \cdot  \pi^*H_{n-1}$. Assume that there is a saturated subsheaf $\cG$ of $\pi^*T(X,D)$ such that  
\begin{enumerate}
\item $\cG$ is $G$-invariant, where $G$ is the Galois group of $\pi$.
\item $\mu_{\alpha,min}(\cG) > 0$,
\item $2\mu_{\alpha,min}(\cG) >  \mu_{\alpha,max} (\pi^*T(X, D) /\cG).$
\end{enumerate}
Then the saturation of $\cG$ in $(\pi^*T_X)^{**}$ defines  an algebraically integrable foliation  $\cF$ on $X$.  Moreover, $\cG$ is the saturation of $(\pi^*\cF)^{**}\cap \pi^*T(X,D)$ in $\pi^*T(X,D)$.
\end{thm}

\begin{proof}
Let $r \colon X'\to X$ be a log resolution of $(X,D)$ which is an isomorphism over the smooth locus $U$ of $(X,D)$.  Let $Z'$ be the normalization of $Z\times_X X'$. Then the natural morphism $\pi' \colon Z'\to X'$ is an adapted finite cover of $(X',D')$, where $D'=r_*^{-1}D$.  Let $U'=r^{-1}(U)$, $V=\pi^{-1}(U)$ and $V'=\pi'^{-1}(U')$.

\centerline{
\xymatrix{
  Z' \ar[d]_{\pi'}   \ar[r] & Z \ar[d]^{\pi}\\
 X'  \ar[r]^{r}  &  X
}
}

There is a unique $G$-invariant saturated subsheaf $\cG'$ of $\pi'^*T(X',D')$  such that $\cG'|_{V'}$ is isomorphic to $\cG|_V$.  
Let $$\alpha' = (\pi'\circ r)^*H_1\cdot \cdots \cdot (\pi'\circ r)^*H_{n-1}.$$ We have $\mu_{\alpha',min}(\cG')= \mu_{\alpha,min}(\cG) > 0$  and  $$\mu_{\alpha' ,max} (\pi'^*T(X', D')/\cG') =  \mu_{\alpha,max} (\pi^*T(X, D)/\cG).$$ 
Hence by \cite[Theorem 1.4]{CP19}, the saturation of $\cG'$ in $\pi'^*T_{X'}$ defines an  algebraically integrable  foliation  $\cF'$ on $X'$. 
Moreover,  $\cG'$ is the saturation of $\pi'^*\cF'\cap \pi'^*T(X',D')$ in $\pi'^*T(X',D')$ (see \cite[Corollary 5.9]{CP19}).  
Let $\cF$ be the saturation of the natural image of $r_*\cF'$ in $T_X$. 
Then $\cF$ is an  algebraically integrable  foliation and $\cG$ is the saturation of $(\pi^*\cF)^{**}\cap \pi^*T(X,D)$ in $\pi^*T(X,D)$.
\end{proof}

Now we will prove Theorem \ref{thm-main-orbifold}.

\begin{proof}[{Proof of Theorem \ref{thm-main-orbifold}}]
The proof is similar to the one of Theorem \ref{thm-main}. Assume the opposite.  
Then there are very ample divisors $H_1,...,H_{n-1}$ such that $T(X,D)$ is not $\pi$-generically semipositive with respect to $H_1,...,H_{n-1}$.  Let $\alpha$ be the class of $\pi^*H_1\cdot \cdots \cdot \pi^*H_{n-1}$.  

By applying Lemma \ref{lem-invol-subsheaf} to $\pi^*T(X,D)$, we can find a saturated subsheaf $\cG$ of $\pi^*T(X,D)$ such that  $\mu_{\alpha,min}(\cG)>0$ and that $ \mu_{\alpha}(\pi^*T(X,D)/\cG) <0$. 
From the uniqueness of Harder-Narasimhan filtration, we know that $\cG$, as a component in the  Harder-Narasimhan filtration, is invariant under the Galois group of $\pi$.  
Moreover, as in the proof of Theorem \ref{thm-main-mov}, we have $$2\mu_{\alpha,min}(\cG) > \mu_{\alpha, max}(\pi^*T(X,D)/\cG).$$   
By Theorem \ref{thm-CP-alg-foliation-sing-orbifold}, the saturation of $\cG$ in $(\pi^*T_X)^{**}$ defines an  algebraically integrable  foliation $\cF$ on $X$.  
Assume that $\cF$ is the relative tangent sheaf of some dominant rational map $f \colon X\dashrightarrow Y$.  

Then, on the one hand, by Theorem \ref{thm-gen-positive-base-anticanonical}, we have $$\alpha\cdot \pi^*(K_{\cF}-K_{X}-D_{ver})\geqslant 0,$$  where $D_{ver}$ is the  vertical part of $D$ over $Y$. 

On the other hand, by Theorem \ref{thm-CP-alg-foliation-sing-orbifold},  the sheaf $\cG$ is  the  saturation of the intersection $(\pi^*\cF)^{**} \cap \pi^*T(X,D)$ in $\pi^*T(X,D)$.   
By \cite[Prop. 2.17]{Clau17}, we have $$\mathrm{det}\,  \cG \cong \sO_Z(\pi^*(-K_{\cF}-D_{hor})),$$ where $D_{hor}$ is the horizontal part of $D$ over $Y$. Thus $$\mathrm{det}\,  (\pi^*T(X,D)/\cG) \cong \sO_Z(\pi^*(K_{\cF}-K_{X}-D_{ver})),$$  and we have $$  \alpha\cdot c_1(\pi^*T(X,D)/\cG) = \alpha\cdot \pi^*(K_{\cF}-K_{X}-D_{ver}) \geqslant 0.$$ 
Since $\mu_{\alpha}(\pi^*T(X,D)/\cG) <0$ by the construction of $\cG$, we obtain a contradiction.
\end{proof}

\section{Relation with rational connectedness}
\label{section:rc}
The aim of this section is to prove Theorem \ref{thm-gen-ample}. We will need the following lemma. 

\begin{lemma}
\label{lem-vanishing}
Let $X$ be a smooth projective variety and let $\delta$ be a cycle of pure dimension $k$. Assume that for any ample divisors $H_1,...,H_k$, we have $\delta\cdot H_1\cdot \cdots \cdot H_k\geqslant 0$. Then the following two conditions are equivalent:
\begin{enumerate}
\item there are ample divisors $H_1,...,H_k$ such that  $\delta\cdot H_1\cdot \cdots \cdot H_k= 0$;
\item for any ample divisors $H_1,...,H_k$, we have $\delta\cdot H_1\cdot \cdots \cdot H_k= 0$;
\item for any nef divisors $H_1,...,H_k$, we have $\delta\cdot H_1\cdot \cdots \cdot H_k= 0$.
\end{enumerate}
\end{lemma}

\begin{proof}
Since every nef divisor is a limit of ample ($\mathbb{Q}$-)divisors, by continuity, we see that (2) implies (3). Hence  we only need  to prove that (1) implies (2). Let $H_1,...,H_k$  be ample divisors such that  $\delta\cdot H_1\cdot \cdots \cdot H_k= 0$, and let $A_1,..., A_k$ be any ample divisors. We need to prove that $\delta\cdot A_1\cdot \cdots \cdot A_k= 0$.

Let $m>0$ be a natural number such that $mH_1-A_1$ is still an ample divisor. Then we have 
\begin{eqnarray*}
0 &\leqslant& \delta\cdot (mH_1-A_1)\cdot \cdots \cdot H_k \\
&=& m(\delta\cdot H_1\cdot \cdots \cdot H_k) - (\delta\cdot A_1\cdot H_2 \cdot \cdots \cdot H_k)\\
&=& -\delta\cdot A_1\cdot H_2\cdot \cdots \cdot H_k \\
&\leqslant& 0.
\end{eqnarray*} Thus $\delta\cdot A_1\cdot H_2\cdot  \cdots \cdot H_k =0.$  By repeating this procedure $k-1$ more times, we can obtain that $\delta\cdot A_1\cdot \cdots \cdot A_k =0.$
\end{proof}

Now we can prove Theorem \ref{thm-gen-ample}.

\begin{proof}[{Proof of Theorem \ref{thm-gen-ample}}]
First we assume that $X$ is rationally connected.  Assume by contradiction that $T_X$ is not generically ample. Then there is a Mehta-Ramanathan-general curve  $C$ such that $T_X|_C$ is nef but not ample.  By \cite[Theorem 2.4]{Har71}, there is a non-zero quotient bundle of $T_X|_C$ of degree zero.  Thus, by Mehta-Ramanathan theorem (see \cite[Theorem 6.1]{MR82}), we have $\mu_{\alpha,min}(T_X)=0$, where $\alpha$ is the class of $C$. 
This implies that there is a surjective morphism $T_X\to Q$ such that $Q$ is a non-zero torsion-free sheaf and that $\alpha \cdot c_1(Q)=0.$ 
We note that $Q$ is generically nef as well  and hence $c_1(Q)\cdot H_1\cdot \cdots \cdot H_{n-1} \geqslant 0$ for any ample divisors $H_1,...,H_{n-1}$. Thus by Lemma \ref{lem-vanishing}, $c_1(Q)\cdot H_1\cdot \cdots \cdot H_{n-1} = 0$ for any ample divisors $H_1,...,H_{n-1}$. This shows that $c_1(Q)$ is numerically zero. Since $X$ is rationally connected, it is simply connected. Therefore, we have $\mathrm{det}\, Q \cong \sO_X.$ The injective  morphism $Q^*\to \Omega_X^1$ then induces a non-zero morphism $\sO_X\to \Omega_X^k$, where $k$ is the rank of $Q$. This is a contradiction, since $h^0(X,\Omega_X^k)=0$ (see \cite[Corollary IV.3.8]{Kol96}).

Now we assume that $X$ is not rationally connected. Then, by \cite[Corollary 1]{Zhang05}, there is a dominant rational map $f \colon X\dashrightarrow Y$ such that $Y$ is smooth with Kodaira dimension $\kappa(Y)=0$ and that  general fibers of $f$ are proper,  rationally connected. Since $X$ is not rationally connected, $Y$ has positive dimension $d$. There is some positive integer $m$ such that $h^0(Y, \sO_Y(mK_Y))\neq 0$. 
Thus $h^0(X, (\Omega_X^1)^{\otimes md}) \neq 0.$ Let $C$ be a Mehta-Ramanathan-general curve and let $\alpha$ be the class of $C$. Then $\mu_{\alpha, max}(\Omega_X^1) \geqslant 0$ by \cite[Corollary 5.11]{CP11}. This implies that $\mu_{\alpha, min}(T_X)  \leqslant 0$. Hence  $T_X|_C$ is not ample by  Mehta-Ramanathan theorem.
\end{proof}

\begin{rem}
\label{rem-sing-nef-rc-gn}
Theorem \ref{thm-gen-ample} does not hold without assuming the smoothness of $X$. For example, let $G$ be the group $\mathbb{Z}/2\mathbb{Z}$ and let $E$ be an elliptic curve with an action of $G$ such that $E/G=\p^1$. We also endow $\p^1$ with the canonical action of $G$ ($g.[a:b]=[a:-b]$ if $g$ is the generator of $G$). 
Then $G$ acts on the product $E\times \p^1$ diagonally and the quotient map $E\times \p^1 \to (E\times \p^1)/G=X$ is \'etale in codimension $1$. 
We see that  $X$ is singular and  $-K_X$ is nef. In addition, as in   \cite[Remark and Question 3.8]{GKP14}, we have $h^0(X, ((\Omega_X^1)^{\otimes 2})^{**}) > 0.$ This implies that $\mu_{\alpha, min}(T_X) \leqslant 0$ for any ample class $\alpha$.
\end{rem}

\section{Equality conditions of Miyaoka inequality}
\label{section:Miyaoka}

As mentioned before the proof of Corollary \ref{cor-c2}, a non-zero torsion-free sheaf $E$ on a projective manifold $X$ of dimension $n$ is said to be generically $(H_1,...,H_{n-2})$-semipositive for some ample divisors $H_1,...,H_{n-2}$ if for each nef divisor $D$, we have $\mu_{\alpha,min}(E)\geqslant 0$, where $\alpha$ is the class of $D\cdot H_1\cdot \cdots \cdot H_{n-2}$.   

In \cite[Theorem 6.1]{Miy87}, Miyaoka proved that if  $c_1(E)$ is nef and if $E$  is generically $(H_1,...,H_{n-2})$-semipositive, then $c_2(E)\cdot H_1\cdot \cdots \cdot H_{n-2}\geqslant 0$.  As a consequence,  if $c_1(E)$ is nef and if $E$ is generically nef, then $c_2(E)\cdot H_1\cdot \cdots \cdot H_{n-2}\geqslant 0$  for any  ample divisors $H_1,...,H_{n-2}$.  

In this section, we will study the equality conditions of these inequalities. 
This is crucial for the proof of Theorem \ref{thm-classification}.  
We will assume that  $E$ is generically nef and that $c_1(E)$ is nef. By  Lemma \ref{lem-vanishing},  the equality $c_2(E)\cdot H_1\cdot \cdots \cdot H_{n-2}= 0$  holds for some  ample divisors $H_1,...,H_{n-2}$ if and only if $c_2(E)\cdot H_1\cdot \cdots \cdot H_{n-2}= 0$  for any ample divisors $H_1,...,H_{n-2}$. Thus, in order to study the equality conditions, we may assume that $c_2(E)\cdot H^{n-2}=0$ for some ample divisor $H$.

Our idea is to look into the details in Miyaoka's proof, and study every inequality inside. We will discuss  following  the numerical dimension $\nu(c_1(E))$  of $c_1(E)$. Recall that the numerical dimension $\nu$ of a nef divisor $N$  is the largest integer such that $N^{\nu+1} \equiv 0$.

\subsection{Case of $\nu(c_1(E)) \geqslant 2$} We first consider the case when $c_1(E)$ has numerical dimension at least $2$, and  show the following proposition.

\begin{prop}
\label{prop-HN-E-ne-0-nd>1}
Let $X$ be a smooth projective variety of dimension $n\geqslant 2$.  Let  $E$ be a non-zero torsion-free sheaf of rank at least $2$ on $X$ which is generically $(H_1,...,H_{n-2})$-semipositive for some ample divisors $H_1,...,H_{n-2}$.  Assume that $c_1(E)$ is nef with numerical dimension at least $2$. Let $\alpha$ be the class of $c_1(E)\cdot H_1\cdots H_{n-2}$.  If $c_2(E)\cdot H_1\cdot \cdots \cdot H_{n-2}=0$, then the Harder-Narasimhan filtration of $E$ with respect to $\alpha$ is of the form  $$0\subsetneq E_1 \subsetneq E,$$  such that 
\begin{enumerate}
\item $E_1$ is an invertible sheaf such that $c_1(E_1)\equiv c_1(E)$;
\item $c_1(E/E_1)\equiv 0;$
\item $c_2(E/E_1)\cdot H_1\cdot \cdots \cdot H_{n-2}=0.$
\end{enumerate}
\end{prop}

For the proof of the proposition, we need the following lemma.

\begin{lemma}
\label{lem-c2-quotient=0}
Let $X$ be a smooth projective variety of dimension $n\geqslant 2$. Let $E$ be a non-zero torsion-free sheaf on $X$ which is generically $(H_1,...,H_{n-2})$-semipositive for some ample divisors $H_1,...,H_{n-2}$. Assume that $c_1(E)$ is nef and  $$c_2(E)\cdot H_1\cdot \cdots \cdot H_{n-2}=0.$$  Let $0\to F\to E\to Q\to 0$ be an exact sequence of non-zero torsion-free sheaves. If $c_1(Q)\equiv 0$, then  $F$ is generically $(H_1,...,H_{n-2})$-semipositive, and  $$c_2(Q)\cdot H_1\cdot \cdots \cdot H_{n-2}=c_2(F)\cdot H_1\cdot \cdots \cdot H_{n-2}=0.$$
\end{lemma}

\begin{proof}
 Since $c_1(Q)\equiv 0$, we have $c_1(E/G)\equiv c_1(F/G)$ for any saturated subsheaf $G$ of $F$. This implies that $F$ is  generically $(H_1,...,H_{n-2})$-semipositive and $c_1(F)\equiv c_1(E)$ is nef. We have $$c_2(E)=c_2(F)+c_2(Q)+c_1(F)\cdot c_1(Q)\equiv c_2(F)+c_2(Q).$$ We note that $Q$ is also generically $(H_1,...,H_{n-2})$-semipositive.   By Miyaoka inequality (see \cite[Theorem 6.1]{Miy87}),  $c_2(F) \cdot H_1\cdot \cdots \cdot H_{n-2} \geqslant 0$ and $c_2(Q) \cdot H_1\cdot \cdots \cdot H_{n-2} \geqslant 0.$ We obtain that $$0= c_2(E)\cdot H_1\cdot \cdots \cdot H_{n-2}=c_2(F) \cdot H_1\cdot \cdots \cdot H_{n-2} +c_2(Q) \cdot H_1\cdot \cdots \cdot H_{n-2} \geqslant 0.$$  This implies that  $$c_2(Q)\cdot H_1\cdot \cdots \cdot H_{n-2}=c_2(F)\cdot H_1\cdot \cdots \cdot H_{n-2}=0.$$
\end{proof}

\begin{proof}[{Proof of Proposition  \ref{prop-HN-E-ne-0-nd>1}  }]

 Since $c_1(E)$ is nef and has numerical dimension at least $2$,  $c_1(E)^2\cdot H_1 \cdot \cdots \cdot  H_{n-2}>0.$   In particular, the class $\alpha$ is not zero. Since $c_2(E)\cdot H_1\cdot \cdots \cdot H_{n-2}=0$, and since $E$ has rank at least $2$, Bogomolov-Gieseker inequality for semistable sheaves (see \cite[Corollary 4.7]{Miy87}) shows that $E$ is not $\alpha$-semistable.   Let $$0=E_0\subsetneq E_1 \subsetneq \cdots \subsetneq E_r=E $$ be the Harder-Narasimhan semistable filtration. We note that (1) implies (2),  and hence  (3) by Lemma \ref{lem-c2-quotient=0}.  
Therefore,  we only need to prove  that the  filtration has length $r=2$ and satisfies  property (1).

Without loss of generality, we may assume that $H_1,...,H_{n-2}$ are effective very ample divisors in general position. Let $S$ be the intersection surface of $H_1,...,H_{n-2}$.  For simplicity, we let $G_i=E_i/E_{i-1}$ and $r_i=\mathrm{rank}\, G_i$ for $i=1,...,r$.  We have 
\begin{eqnarray*}
2c_2(E)\cdot S &=&  (\sum_{i=1}^r 2c_2(G_i) + \sum _{1\leqslant i < j \leqslant r}  2c_1(G_i) c_1 (G_j))\cdot S\\
                       &= &  (\sum_{i=1}^r 2c_2(G_i) +  c_1(E)^2 -\sum _{ i  =1}^r   c_1(G_i)^2 )\cdot S.
\end{eqnarray*}

Since each $G_i$ is $\alpha$-semistable, Bogomolov-Gieseker inequality   (see \cite[Corollary 4.7]{Miy87}) shows that 
\begin{eqnarray*} 
2c_2(E)\cdot S &\geqslant&  (\sum_{i=1}^r \frac{r_i-1}{r_i} c_1(G_i)^2 +  c_1(E)^2 -\sum _{ i  =1}^r   c_1(G_i)^2)\cdot S\\
  &=& (  c_1(E)^2 -\sum _{ i  =1}^r  \frac{1}{r_i} c_1(G_i)^2)\cdot S.     
\end{eqnarray*}

Since $c_1(E)^2\cdot S >0$,   Hodge index theorem on $S$ shows that $$(c_1(E)^2\cdot S )(c_1(G_i)^2 \cdot S)\leqslant ( c_1(E) \cdot c_1(G_i) \cdot S)^2.$$  Therefore, 
\begin{equation}\label{eq:hodge:index} 
c_1(G_i)^2 \cdot S \leqslant   \frac{ (c_1(E) \cdot c_1(G_i) \cdot S)^2}{c_1(E)^2\cdot S }.
\end{equation}
We let $$a_i =  \frac{ c_1(E) \cdot c_1(G_i) \cdot S }{r_i c_1(E)^2\cdot S}.$$ By the definition of Harder-Narashimhan filtration, we see that $a_1>\cdots > a_r$.   Moreover, we note that $\sum_{i=1}^r r_i a_i=1$ and $a_i\geqslant 0$ for all $i$.  Hence $a_i\leqslant 1$ for all $i$.  The inequality (\ref{eq:hodge:index}) becomes  $$c_1(G_i)^2 \cdot S \leqslant   r_i^2a_i^2  c_1(E)^2\cdot S.$$  Therefore, we have  
\begin{eqnarray*}
2c_2(E)\cdot S &\geqslant& (  c_1(E)^2 -\sum _{ i  =1}^r  \frac{1}{r_i} c_1(G_i)^2)\cdot S  \\
                        &\geqslant& (  c_1(E)^2 -\sum _{ i  =1}^r  \frac{1}{r_i} \cdot  r_i^2a_i^2  c_1(E)^2 )\cdot S \\
                        &=&   (  c_1(E)^2 -\sum _{ i  =1}^r     r_ia_i^2  c_1(E)^2 )\cdot S\\
                        &=&(1- \sum _{ i  =1}^r     r_ia_i^2) c_1(E)^2\cdot S\\
                        &\geqslant& (1- \sum _{ i  =1}^r     r_ia_ia_1) c_1(E)^2\cdot S\\
                        &=& (1- a_1) c_1(E)^2\cdot S.
\end{eqnarray*}

By assumption,  $c_2(E)\cdot S=0$ and  $c_1(E)^2\cdot S>0$. Since $a_1\leqslant 1$, the inequality above shows that  $a_1=1$. We recall that $\sum_{i=1}^r r_ia_i=1$, $a_i\geqslant 0$ for all $i$ and $a_1>\cdots > a_r$. Hence we can only have $r=2$, $r_1=1$ and $a_2=0$.

It remains to prove that $c_1(E/E_1)\equiv 0$.  On the one hand,  we have  $$c_1(E/E_1)\cdot c_1(E) \cdot S =  c_1(E)^2\cdot S - c_1(E_1)\cdot c_1(E)\cdot S = 0.$$  Since   $c_2(E_1)=0$, we obtain that
\begin{eqnarray*}
 0 = c_2(E) \cdot S &=& (c_1(E_1)\cdot c_1(E/E_1) + c_2(E/E_1))\cdot S \\ 
 &=& (c_1(E)\cdot c_1(E/E_1)-c_1(E/E_1)^2+c_2(E/E_1))\cdot S\\
 &=& (-c_1(E/E_1)^2+c_2(E/E_1))\cdot S.
\end{eqnarray*}
Since $E/E_1=G_2$ is semistable, Bogomolov-Gieseker inequality shows that $$c_2(E/E_1)\cdot S \geqslant \frac{r_2-1}{2r_2} c_1(E/E_1)^2 \cdot S.$$  Therefore, we have   $$0=(-c_1(E/E_1)^2+c_2(E/E_1))\cdot S \geqslant (-1+\frac{r_2-1}{2r_2}) c_1(E/E_1)^2 \cdot S,$$ which implies that $c_1(E/E_1)^2 \cdot S \geqslant 0.$

On the other hand, from  Lefschetz theorem, we see that  the symmetric bilinear form   $\mathbf{q}(\delta, \delta')=\delta \cdot \delta' \cdot S$ defined on $N^1(X)$, the space of real numerical divisors classes, is nondegenerate. By Hodge index theorem,   $\mathbf{q}$ has exactly one positive eigenvalue. We also have $\mathbf{q}(c_1(E),c_1(E))>0$, and $\mathbf{q}(c_1(E),c_1(E/E_1))=0.$ Hence, by Sylvester theorem,  $$c_1(E/E_1)^2\cdot S= \mathbf{q}(c_1(E/E_1),c_1(E/E_1))\leqslant 0,$$ and the equality holds if and only if $c_1(E/E_1)\equiv 0$.  This completes the proof of the proposition.
\end{proof}

\subsection{Case of $\nu(c_1(E)) =1$}

Next we will consider the case when $\nu(c_1(E)) = 1$.  In this case, $E$ is semistable with respect to the class $c_1(E)\cdot H^{n-2}$, and the Harder-Narasimhan filtration with respect to this class   does not  provide any further information. 
Instead, we will use curve classes of the form  $(c_1(E)+ \epsilon H)^{n-1}$ with $\epsilon>0$, and  prove the following  proposition.

\begin{prop}
\label{prop-HN-E-ne-nd=1}
Let $X$ be a smooth projective variety of dimension $n\geqslant 2$.  Let  $E$ be a non-zero torsion-free sheaf such  that $c_1(E)$ is nef with numerical dimension    $1$, and that $E$ is generically  nef. Assume that $c_2(E)\cdot H^{n-2}=0$ for some ample divisor $H$.  Then there is a filtration   $$0=E_0\subsetneq E_1 \subsetneq \cdots  \subsetneq E_r= E$$  such that 
\begin{enumerate}
\item for each $\epsilon>0$ small enough, it is the Harder-Narasimhan filtration with respect to the class $$\alpha_\epsilon=(c_1(E)+\epsilon H)^{n-1}.$$   
\item  $c_1(E_k/E_{k-1})$ is nef and numerically proportional to  $c_1(E)$  for any $k=1,...,r$;
\item   $c_2(E_k/E_{k-1})\cdot (c_1(E)+\epsilon H)^{n-2}=0$ for any $k=1,...,r$ and any $\epsilon>0$ small enough.
\end{enumerate}
\end{prop}

The proof of the proposition consists of several steps. The existence of common  Harder-Narasimhan filtration follows from the following lemma.

\begin{lemma}
\label{lem-common-filtration}
Let $X$ be a normal projective $\mathbb{Q}$-factorial variety of dimension $n\geqslant 2$. Let $D_1,...,D_{n-1}$ be nef divisors   and let $H$ be an ample divisors.  Let $E$ be a torsion-free sheaf on $X$. Then there is  a filtration $$0=E_0\subsetneq E_1 \subsetneq \cdots  \subsetneq E_r= E$$ such that  for any $\epsilon>0$ small enough, it is the Harder-Narasimhan filtration with respect to the class  of $(D_1+{\epsilon} H) \cdot \cdots \cdot (D_{n-1}+{\epsilon} H)$.
\end{lemma}

\begin{proof}
See \cite[Lemma 6.5]{KMMc04}.
\end{proof}

The second and the third properties in Proposition \ref{prop-HN-E-ne-nd=1} are consequences of the two lemmas below.

\begin{lemma}
\label{lem-filtration-c2}
Let $X$ be a smooth projective variety of dimension $n\geqslant 2$.  Let  $E$ be a non-zero torsion-free sheaf. Let $ H_1,...,H_{n-2}$ be ample divisors and let $D$ be a nef divisor such that the class $\alpha$ of $D\cdot H_1\cdot \cdots \cdot H_{n-2}$ is not zero. Assume that
\begin{enumerate}
\item[(i)] $c_1(E)^2\cdot H_1 \cdot \cdots \cdot H_{n-2}=c_2(E)\cdot H_1 \cdot \cdots \cdot H_{n-2}=0;$
\item[(ii)] there is a filtration  of saturated subsheaves $$0=E_0\subsetneq E_1 \subsetneq \cdots  \subsetneq E_r= E$$ such that $E_k/E_{k-1}$ is $\alpha$-semistable and that $$c_1(E_k/E_{k-1})^2\cdot H_1 \cdot \cdots \cdot H_{n-2} \leqslant 0 $$ for each $k=1,...,r$.
\end{enumerate}
Then  we have 
\begin{eqnarray*}
c_1(E_k/E_{k-1})^2 \cdot H_1 \cdot \cdots \cdot H_{n-2} = c_2(E_k/E_{k-1}) \cdot H_1 \cdot \cdots \cdot H_{n-2} =0.
\end{eqnarray*}

\end{lemma}

\begin{proof}
Let $S=H_1 \cdot \cdots \cdot H_{n-2}$. Then  we have $$0=2c_2(E)\cdot S = (\sum_{k=1}^r 2 c_2(E_k/E_{k-1}) + \sum_{1\leqslant j < k \leqslant r} 2 c_1(E_j/E_{j-1}) \cdot c_1(E_{k}/E_{k-1})) \cdot S.$$ 
Since each $E_k/E_{k-1}$ is also $\alpha$-semistable, by Bogomolov-Gieseker inequality,  we deduce that
\begin{eqnarray*}
0&\geqslant&  (\sum_{k=1}^r \frac{d_k-1}{d_k} c_1(E_k/E_{k-1})^2 + \sum_{1\leqslant j < k \leqslant r} 2 c_1(E_j/E_{j-1}) \cdot c_1(E_{k}/E_{k-1})) \cdot S\\
&=&(c_1(E)^2-\sum_{k=1}^r\frac{1}{d_k}c_1(E_k/E_{k-1})^2)\cdot S \\
&=& -\sum_{k=1}^r(\frac{1}{d_k}c_1(E_k/E_{k-1})^2 \cdot S) \\
&\geqslant & 0,
\end{eqnarray*}
where $d_k$ is the rank of $E_k/E_{k-1}$. Therefore, all of the inequalities above are equalities. We conclude hence  $$ c_1(E_k/E_{k-1})^2 \cdot S =0, $$ and  $$c_2(E_k/E_{k-1}) \cdot S = \frac{d_k-1}{2d_k} c_1(E_k/E_{k-1})^2 \cdot S =0,$$ for each $k=1,...,r$. 
\end{proof}

\begin{lemma}
\label{lem-HN-E-ne-nd=1}
Let $X$ be a smooth projective variety of dimension $n\geqslant 2$.  Let  $E$ be a non-zero torsion-free sheaf such   that $c_1(E)$ is nef with numerical dimension   $1$ and that $c_2(E)\cdot H^{n-2}=0$ for some ample an ample divisor $H$.   
Assume that there is a filtration  of saturated subsheaves $$0=E_0\subsetneq E_1 \subsetneq \cdots  \subsetneq E_r= E$$ which satisfies the following two conditions:
\begin{enumerate}
\item[(i)] for each $\epsilon>0$ small enough and each $k=1,...,r$,  the sheaf $E_k/E_{k-1}$ is semistable with respect to the class $$\alpha_\epsilon=(c_1(E)+\epsilon H)^{n-1};$$
\item[(ii)]  for each $\epsilon>0$ small enough and each $k=1,...,r$, $$\mu_{\alpha_\epsilon}(E_k/E_{k-1})\geqslant  0.$$
\end{enumerate} 
Then  $c_1(E_k/E_{k-1})$ is nef and numerically proportional to  $c_1(E)$  for any $k=1,...,r$ .
\end{lemma}

\begin{proof}
We can   assume that $H$ is very  ample.  Let $S$ be the surface  cut out by general members of the linear system of $H$.  
 Since $c_1(E)$ has numerical dimension   $1$, we have $c_1(E)^i\cdot H^{n-i}=0$ for $i\geqslant 2$. 
 Thus, for each $\epsilon>0$, $$\alpha_\epsilon =  \binom{n-1}{2}\epsilon^{n-2} \cdot (c_1(E) +  \frac{\epsilon}{\binom{n-1}{2}} H )\cdot S.$$ 
To simplify, we let $\beta_{\eta}$ be the curve class $$\beta_\eta = (c_1(E)+\eta H)\cdot S$$ for any $\eta >0$. 
Then for any $\eta>0$ small enough and any $k=1,...,r$, condition (i) implies that $E_k/E_{k-1}$ is semistable with respect to  $\beta_\eta$  and condition (ii) implies that $\mu_{\beta_\eta}(E_k/E_{k-1})\geqslant 0.$ 
 In particular, if we tend $\eta$   to zero, then we obtain that  $$c_1(E_k/E_{k-1})\cdot c_1(E) \cdot S \geqslant 0$$ for any $k=1,...,r$. Since $c_1(E)=\sum_{k=1}^r c_1(E_k/E_{k-1})$ and $c_1(E)^2\cdot S = 0$, this shows that $$c_1(E_k/E_{k-1})\cdot c_1(E) \cdot S=0. $$

From  Lefschetz theorem,  
the symmetric bilinear form   $\mathbf{q}(\delta, \delta')=\delta \cdot \delta' \cdot S$ defined on $N^1(X)$ is nondegenerate, where $N^1(X)$ is the space of real numerical divisors classes. By Hodge index theorem, $\mathbf{q}$ has exactly one positive eigenvalue. Since $c_1(E)\not\equiv 0$,  the condition $$\mathbf{q}(c_1(E_k/E_{k-1}),c_1(E))=\mathbf{q}(c_1(E),c_1(E))=0$$  implies that $\mathbf{q}(c_1(E_k/E_{k-1}),c_1(E_k/E_{k-1}))\leqslant 0$ by Sylvester theorem.

Since each $E_k/E_{k-1}$ is semistable with respect to the class $\alpha_\epsilon$ for $\epsilon>0$ small enough, by Lemma \ref{lem-filtration-c2},  we obtain that $$\mathbf{q}(c_1(E_k/E_{k-1}),c_1(E_k/E_{k-1}))=0.$$ Thus, by Lemma \ref{lem-2form-1} below, $c_1(E_k/E_{k-1})$ is numerically proportional to $c_1(E)$ for any $k=1,...,r$. 
Moreover, it is nef since $\mu_{\alpha_\epsilon}(E_k/E_{k-1}) \geqslant  0$ for $\epsilon>0$ small enough.   
\end{proof}

\begin{lemma}
\label{lem-2form-1}
Let $\mathbf{q}(\cdot,\cdot)$ be a nondegenerate symmetric bilinear form of signature $(1,m)$ on a real vector space $V$. Let $\vec{x},\vec{y}\in V$ be two  vectors. Assume that $\mathbf{q}(\vec{x},\vec{x})=0$, $\mathbf{q}(\vec{x},\vec{y})=0$ and $\mathbf{q}(\vec{y},\vec{y})=0$. Then $\vec{x}$ and $\vec{y}$ are linear dependent.
\end{lemma}

\begin{proof}
If $m=0$, then $\mathbf{q}$ is definite and $\vec{x}=\vec{y}=\vec{0}$. We assume then that $m>0$.    By Sylvester theorem, there is an orthogonal basis $(\vec{e},\vec{e}_1,...,\vec{e}_m)$ of $V$ such that $\mathbf{q}(\vec{e},\vec{e})=1$ and that $\mathbf{q}(\vec{e}_i,\vec{e}_i)=-1$ for all $i=1,...,m$.

Let $(a,b_1,...,b_m)$ and $(a',b_1',...,b_m')$ be the coordinates of $\vec{x}$ and $\vec{y}$ respectively. Then by assumption, we have $$a^2=b_1^2+\cdots b_m^2,\ a'^2=b_1'^2+\cdots b_m'^2,\ \mathrm{and}\ aa'=b_1b_1'+\cdots b_mb_m'.$$ Thus we have $$(aa')^2=(b_1b_1'+\cdots +b_mb_m')^2=(b_1^2+\cdots b_m^2)(b_1'^2+\cdots b_m'^2).$$ From the equality condition of Cauchy inequality, this shows that $(b_1,...,b_m)$ and $(b_1',...,b_m')$ are linearly dependent. We then deduce that $(a,b_1,...,b_m)$ and $(a',b_1',...,b_m')$ are linearly dependent.
\end{proof}

Now we can deduce Proposition \ref{prop-HN-E-ne-nd=1}.

\begin{proof}[{Proof of Proposition \ref{prop-HN-E-ne-nd=1}}]
Let $$0=E_0\subsetneq E_1 \subsetneq \cdots  \subsetneq E_r= E$$ be the common Harder-Narasimhan filtration with respect the classes $\alpha_\epsilon$  for all $\epsilon>0$ small enough (see  Lemma \ref{lem-common-filtration}).  Then this filtration satisfies the property (1). Since $E$ is generically nef, $\mu_{\alpha_\epsilon, min }(E) \geqslant 0$ for any $\epsilon>0$. Thus  for each $\epsilon>0$ small enough and each $k=1,...,r$, we have $$\mu_{\alpha_\epsilon}(E_k/E_{k-1}) \geqslant \mu_{\alpha_\epsilon}(E_r/E_{r-1})=\mu_{\alpha_\epsilon, min }(E) \geqslant 0.$$ 
The property (2) then follows from Lemma \ref{lem-HN-E-ne-nd=1}. In particular, we deduce that $c_1(E_k/E_{k-1})^2\cdot H^{n-2}  = 0 $ for each $k=1,...,r$. Thus the  property (3) follows from Lemma \ref{lem-filtration-c2}. 
\end{proof}

For each semistable component $E_k/E_{k-1}$ in Proposition \ref{prop-HN-E-ne-nd=1}, we note that the equality holds in the Bogomolov-Gieseker inequality.  It then follows from \cite[Theorem IV.4.1]{Nakayama2004} that the reflexive hull $(E_k/E_{k-1})^{**}$ is a numerically projectively flat vector bundle. 
Particularly, we will need the following property for these sheaves.

\begin{lemma}
\label{lem-BG-equality-cond}
Let  $X$ be a smooth projective variety of dimension $n>1$ and $H$  an ample divisor.  Let $E$ be a $(H^{n-1})$-semistable torsion-free sheaf of rank $r$. Assume that 
\[(c_2(E) - \frac{r-1}{2r} c_1(E)^2)\cdot H^{n-2} = 0.\]
Then for any rational curve $C$ in $X$, the intersection number  $C \cdot c_1(E)$ is an integer divisible by  $r$.
\end{lemma} 
 
\begin{proof}
By Lemma \ref{lem-c_2-decreasing} below,  we may assume that $E$ is reflexive. Then \cite[Theorem IV.4.1]{Nakayama2004} implies that $E$ is locally free. Moreover, there is a filtration of subbundles $$0 =E_0 \subsetneq E_1 \subsetneq \cdots  \subsetneq E_r= E$$ 
such that each quotient $E_i/E_{i-1}$ is a hermitian projectively flat vector bundle. Moreover,  we   have 
\begin{equation}
\label{equation-c_1}
\frac{1}{r_i} c_1(E_i/E_{i-1}) \equiv \frac{1}{r}c_1(E),
\end{equation} 
where $r_i = \mathrm{rank}\ E_i/E_{i-1}$.
We also note that the projectivized bundle $\mathbb{P}(E_i/E_{i-1})$ is induced by some $\mathrm{PGL}$-representation of the fundamental group $\pi_1(X)$ (see \cite[Corollary IV.4.3]{Nakayama2004}). 

Let $\rho\colon \mathbb{P}^1 \to X$ be the morphism induced by the normalization of $C$. Since $\mathbb{P}^1$ is simply connected, it follows that there is a line bundle $L_i$ on $\mathbb{P}^1$ such that  $\rho^*(E_i/E_{i-1}) \cong L_i^{\oplus r_i}$. 
Hence there is an integer $a_i$ such that $$C\cdot c_1(E_i/E_{i-1}) = r_ia_i.$$ 
In addition, the equation (\ref{equation-c_1}) implies that $a_i$ is independent of $i$. We denote this constant integer by $a$, then we have $C\cdot c_1(E) =ra$.
\end{proof}

The following lemma might be well-known to experts. For the  reader's convenience, we recall briefly the proof here.

\begin{lemma}
\label{lem-c_2-decreasing}
Let $X$ be a smooth  projective  variety of dimension $n\geqslant 2$. Let $E$ be a torsion-free sheaf on $X$. Then for any ample divisors $H_1,...,H_{n-2}$, we have $$c_2(E)\cdot H_1\cdot \cdots \cdot H_{n-2} \geqslant c_2(E^{**})\cdot H_1\cdot \cdots \cdot H_{n-2}.$$ Moreover, the equality holds if and only if $E$ is locally free in codimension $2$.
\end{lemma}

\begin{proof}
 We may assume that $H_1,...,H_{n-2}$ are effective sufficiently ample divisors in general positions. Let $S$ be their intersection. Since $X$ is smooth, there is a finite free resolution of $E$ as follows, $$0\to F_k\to\cdots\to F_0\to E.$$ Since $S$ is in general position, we may assume that $E|_S$ is still  torsion-free, and that  $$0\to F_k|_S\to\cdots\to F_0|_S\to E|_S$$ is again a free resolution. Hence $$c_2(E)\cdot H_1\cdot \cdots \cdot H_{n-2} =c_2(E|_S).$$

By the same argument, we may assume that $E^{**}|_S$ is still reflexive and is isomorphic to $(E|_S)^{**}$.  Moreover, we may also assume that $$c_2(E^{**})\cdot H_1\cdot \cdots \cdot H_{n-2} =  c_2(E^{**}|_S) =  c_2((E|_S)^{**}).$$ By \cite[Lemma 10.9]{Meg92}, we have $$c_2((E|_S)^{**}) \geqslant c_2(E|_S),$$ and the equality holds if and only if $E|_S$ is locally free. Since $S$ is in general position, the sheaf $E|_S$ is locally free if and only $E$ is locally free in codimension $2$. This completes the proof of the lemma.
\end{proof}

\subsection{Case of $\nu(c_1(E))= 0$} We will finish this section with  the case when $\nu(c_1(E)) = 0$. 
We have the following result.

\begin{lemma}
\label{lem-HN-E-ne-nd=0}
Let $X$ be a smooth projective variety of dimension $n\geqslant 2$.  Let  $E$ be a non-zero torsion-free sheaf such  that $c_1(E)\equiv 0$  and that $E$ is generically  nef.  
If we assume further that $c_2(E)\cdot H^{n-2}=0$ for some ample divisor $H$, then the reflexive hull $E^{**}$ is a flat vector bundle. 
\end{lemma}

\begin{proof}
We note that $E$ is  semistable with respect to  the class $H^{n-1}$. 
Moreover, the equality holds in the Bogomolov-Gieseker inequality.     
It then follows from \cite[Theorem IV.4.1]{Nakayama2004} that $E^{**}$  is a numerically flat vector bundle (see \cite[Definition 1.17 and Theorem 1.18]{DPS94}). 
Hence it is a flat vector bundle, see \cite[Section 3]{Simpson1992}.
\end{proof}

\section{Proof of the classification theorem}
\label{section:Yau}

We will finish the proof of Theorem \ref{thm-classification} in this section. 

\begin{prop}
\label{prop-c2=0-q<quotient}
Let $X$ be a  smooth projective variety of dimension $n\geqslant 2$ with nef anticanonical class $-K_X$. Assume that $c_2(T_X)\cdot H^{n-2}=0$ for some ample divisor $H$.  Assume further that there is some non-zero torsion-free quotient $T_X\to Q$ such that $c_1(Q)\equiv 0$ and $\mathrm{rank}\, Q=k$. Then  the augmented irregularity $\tilde{q}(X)$ is at least equal to $k$.
\end{prop}

We recall that the irregularity of a smooth projective variety $X$ is $q(X)=h^1(X,\sO_X)$. It is equal to the dimension of the Albanese variety of $X$.   The augmented irregularity is defined as $$\tilde{q}(X) = \sup \{q(\tilde{X})\ | \  \tilde{X}\to X \mbox{ is a finite {\'e}tale cover} \}.$$ We will first prove  that  $\tilde{q}(X)$ is not zero under the assumption of Proposition \ref{prop-c2=0-q<quotient}.

\begin{lemma}
\label{lem-c2=0-q=0}
Under the condition of Proposition \ref{prop-c2=0-q<quotient}, we have  $\tilde{q}(X)\neq 0.$
\end{lemma}

\begin{proof}
Assume by contradiction that $\tilde{q}(X)=0$. Then by \cite[Theorem 2]{Pau97}, $X$ has finite fundamental group. By replacing $X$ by some finite \'etale cover if necessary, we may assume that $X$ is simply connected. 

By Theorem \ref{thm-main}, $T_X$ is generically nef. Thus so is the quotient $Q$.  From Lemma \ref{lem-c2-quotient=0}, we deduce that $c_2(Q)\cdot H^{n-2}=0$.   By Lemma \ref{lem-HN-E-ne-nd=0}, the reflexive hull   $Q^{**}$ is a flat  locally free  sheaf.  Since $X$ is simply connected, this implies that  $Q^{**}$ is isomorphic to a direct sum of copies of $\sO_X$.   We obtain then a non-zero morphism $\sO_X\to \Omega_X^1.$ This shows that $q(X)=h^0(X,\Omega_X^1)\geqslant 1$, which is a contradiction.
\end{proof}

Now we will prove  Proposition \ref{prop-c2=0-q<quotient}.

\begin{proof}[{Proof of Proposition \ref{prop-c2=0-q<quotient}}]
By replacing $X$ by some finite \'etale cover if necessary, we may assume that  the Albanese morphism $f \colon X\to A$ induces an isomorphism of fundamental groups (see \cite[Theorem 2]{Pau97}).   Then   \cite[Theorem 1.2]{Cao19} implies that $f$ is an  isotrivial fibration and every fiber of $f$ is simply connected. As a consequence, we have  $$\tilde{q}(X)=q(X)=\mathrm{dim}\, A=q.$$  

Assume  by contradiction that the proposition does not hold. Then we must have $n\geqslant k > q$.    There is an  exact sequence $$0\to T_{X/A} \to T_X \to f^*T_A \to 0.$$ 
Let $F$ be a  general  fiber of $f$ and let $i \colon F\to X$ be the natural injection. Then $\tilde{q}(F)=0$ as $F$ is simply connected. 

If $\mathrm{dim}\, F=1$, then $q = \mathrm{dim}\ A = n-1$, and we can only have  $n=k=q+1$. 
In particular, $T_X=Q$ and $c_1(X)\equiv 0$.  This shows that $F$ is an elliptic curve, which is a contradiction. Hence we have $\mathrm{dim}\, F \geqslant 2$.

We  note that $F$ is a smooth projective manifold with nef anticanonical class $-K_F$. Hence $T_F$ is generically nef by Theorem \ref{thm-main}.
By restricting the exact sequence above to $F$, we can obtain an exact sequence $$0\to T_F \to T_{X}|_F \to (\sO_F)^{q}\to 0.$$ In particular, we have $c_2(T_F)=i^*c_2(T_X)$. Since $F$ is numerically equivalent to the complete intersection of nef divisors, by Corollary \ref{cor-c2} and Lemma \ref{lem-vanishing}, the vanishing condition on $c_2(T_X)$ implies that  
$$c_2(T_F)\cdot (H|_F)^{n-q-2}=c_2(T_X|_F)\cdot (H|_F)^{n-q-2}=c_2(T_X)\cdot H^{n-q-2}\cdot F=0.$$

Since $F$ is a general fiber, we may assume that  $Q|_F$ is torsion-free and  that $$c_1(Q|_F)=i^* c_1(Q) \equiv 0.$$   
We note that the induced morphism $T_F\to Q|_F$ is non-zero for $k>q$. 
Let $R$ be its image.  
We claim that $c_1(R)\equiv 0.$       
Let $\beta=(H|_F)^{n-1-q}$ be a curve class in $F$. 
Then  $ \mu_{\beta}(R)\geqslant 0$ for $T_F$ is generically nef.   
Moreover, by applying Theorem \ref{thm-main-mov} to the class of  $\beta$ in $X$, we have $$\mu_{\beta, max}(Q|_F)\geqslant \mu_{\beta, min}(Q|_F)\geqslant 0.$$  
However, since $c_1(Q|_F)\equiv 0$,   we must have $$\mu_{\beta, max}(Q|_F)= \mu_{\beta, min}(Q|_F) = \mu_{\beta}(Q|_F)= 0.$$   
Thus the condition $ \mu_{\beta}(R)\geqslant 0$    implies that  $\mu_{\beta}(R)=0$ as $R$ is a subsheaf of $Q|_F$. Since $R$ is a quotient of $T_F$, it is generically nef, and hence  $c_1(R)\equiv 0$ by Lemma \ref{lem-vanishing}. 
Thanks to Lemma \ref{lem-c2=0-q=0}, we obtain that $\tilde{q}(F)>0.$ This is a contradiction as $F$ is simply connected.
\end{proof}

Now we can conclude Theorem \ref{thm-classification}.

\begin{proof}[{Proof of Theorem \ref{thm-classification}}]
Since  the second Chern class of an abelian variety is zero, we see that   (2) implies (1) in the theorem. 
We will  prove  that (1) implies (2).  The case when $n=1$ is trivial. We assume from now on that $n\geqslant 2$.  If $K_X\equiv 0$, then by Beauville-Bogomolov decomposition theorem (see \cite[Th\'eor\`eme 1]{Bea83}), there is a finite \'etale cover $X'\to X$ such that $X'$ is isomorphic to a product of irreducible holomorphic symplectic manifolds, Calabi-Yau manifolds and an abelian variety. The vanishing condition (1) on $c_2(T_X)$ then implies that $X'$ is an abelian variety. 

We will now assume that $K_X\not\equiv 0$. Then $X$ is uniruled.  By replacing $X$ by some finite \'etale cover if necessary, we may assume that  the Albanese morphism $f \colon X\to A$ induces an isomorphism of fundamental groups (see \cite[Theorem 2]{Pau97}).   Then   \cite[Theorem 1.2]{Cao19} implies that $f$ is an  isotrivial fibration and every fiber of $f$ is simply connected. As a consequence, we have  $$\tilde{q}(X)=q(X)=\mathrm{dim}\, A=q.$$   Since $X$ is uniruled,  $f$ is not an isomorphism.  Moreover, the fibers of $f$ are   uniruled.

Let $H$ be an ample divisor in $X$. 
By Lemma \ref{lem-vanishing}, we have $c_2(T_X)\cdot H^{n-2}=0$.  Thanks to Theorem \ref{thm-main-mov}, we know that $T_X$ is generically $(H,...,H)$-semipositive.
  
First we assume that $K_X^2\cdot H^{n-2}\neq 0,$ that is, $-K_X$ has numerical dimension at least $2$. Then by Proposition \ref{prop-HN-E-ne-0-nd>1}, there is a torsion-free quotient $T_X\to Q$ such that $c_1(Q)\equiv 0$ and $\mathrm{rank} \, Q=n-1$. Hence by Proposition \ref{prop-c2=0-q<quotient}, the augmented irregularity of $X$ satisfies $q \geqslant n-1$.  Since we have assume that $f \colon X\to A$ is not an isomorphism, we have $q=n-1$. Then $f \colon X\to A$ is a $\p^1$-bundle.

Now we assume that $K_X^2\cdot H^{n-2} = 0.$ Since $-K_X\not\equiv 0$, this means that $-K_X$ has numerical dimension $1$. Let $$0\subsetneq E_1 \subsetneq \cdots  \subsetneq E_r= T_X$$ be the common Harder-Narasimhan filtration with respect to  $$(-K_X+\epsilon H)^{n-1}$$ for all $\epsilon >0$, as in Proposition \ref{prop-HN-E-ne-nd=1}.

Since $X$ is uniruled, there is an elementary  contradiction $g \colon X\to Y$ of some $K_X$-negative extremal ray $R$. We assume that $R$ is generated by a rational curve $C$ of minimal degree with respect to $-K_X$.  
By Lemma \ref{lem-BG-equality-cond}, for each $k=1,...,r$, there is some integer $a_k$   such that $$C 
\cdot c_1(E_k/E_{k-1}) = l_ka_k,$$ where   $l_k = \mathrm{rank} \, (E_k/E_{k-1})$. 
We note that $a_k\geqslant 0$ for all $k$ since $c_1(E_k/E_{k-1})$ is nef.
Hence we have $$-C\cdot K_X = \sum_{k=1}^r l_ka_k.$$

We will discuss   two cases. In the first case, we assume that 
\begin{enumerate}
\item[-] either there is some $k$ such that $l_k\geqslant 2$ and $a_k \neq0$,
\item[-] or there are $k\neq k'$ such that $a_k \neq 0$  and $a_{k'} \neq 0$. 
\end{enumerate}
The assumptions above imply that $-C\cdot K_X \geqslant 2$, that is, the length  of $R$ is at least $2$.  
Moreover, since  the numerical dimension of $-K_X$ is  $1$, and since $-K_X$ is ample on every fiber of the elementary contraction $g$, we obtain that   every fiber of $g$ has dimension at most $1$.   By \cite[Theorem 1.1]{Wis91}, we have $\mathrm{dim}\, Y< n$. Since every fiber of $g$ has dimension at most $1$, we obtain that $\mathrm{dim}\, Y= n-1$.  In addition, $g$ is a conic bundle and $Y$ is smooth by \cite[Theorem 3]{Ando84}. Since the length of $R$ is at least $2$, we obtain that $g$ is smooth.  This implies that    $g_*(K_X^2)\equiv -4K_Y$ (see \cite[Section 4.11]{Miyanishi83}), and thus  $K_Y\equiv 0$.  
Since $g$ has relative dimension $1$ and since $K_Y\equiv 0$, we have $c_2(T_X)\equiv g^*c_2(T_Y)$. The vanishing condition on $c_2(T_X)$ then shows that $c_2(T_Y)\cdot H'_1\cdots H'_{n-3}=0$ for any ample divisors $H'_1,...,H'_{n-3}$ on $Y$ if $\mathrm{dim}\, Y\geqslant 2$. Hence, there is a finite \'etale cover $Y'\to Y$ such that $Y'$ is an abelian variety. As a consequence, $X\times_Y Y'$ is   a $\p^1$-bundle over an abelian variety.

We will now study the remaining case. The sum  $-C\cdot K_X = \sum_{k=1}^r l_ka_k$   satisfies both of the following two conditions:
\begin{enumerate}
\item[(i)] for all $k$, if $a_k\neq  0$, then $l_k=1$;
\item[(ii)] there is at most one $k$ such that $a_k\neq 0$.
\end{enumerate}
From the definition of Harder-Narasimhan filtration, the condition (ii) above implies that the filtration has length $r\leqslant 2$. Since $K_X\not\equiv 0$, we can only have $r=2$ and $c_1(E/E_1)\equiv 0$. The condition (i) above then implies that $\mathrm{rank}\, E_1=1$. Thus, by Proposition \ref{prop-c2=0-q<quotient}, the augmented irregularity of $X$ is at least $n-1$. We conclude then $f \colon X\to A$ is a $\p^1$-bundle.
\end{proof}


\renewcommand\refname{Reference}
\bibliographystyle{alpha}
\bibliography{references}

\end{document}